\documentclass[11pt]{amsart}
\usepackage{amssymb}
\usepackage{hyperref}

\newtheorem{theorem}{Theorem}[section]
\newtheorem{definition}[theorem]{Definition}
\newtheorem{proposition}[theorem]{Proposition}
\newtheorem{corollary}[theorem]{Corollary}

\newtheorem{lemma}[theorem]{Lemma}

\newtheorem{conjecture}[theorem]{Conjecture}

\begin{document}

\title{Small circulant complex Hadamard matrices of Butson type}


\author{Gaurush Hiranandani}
\address{GH: Senior Undergraduate, Department of Mathematics and Statistics, IIT Kanpur, Kanpur, Uttar Pradesh, India - 208016.  {\tt gaurushh@iitk.ac.in}}


\author{Jean-Marc Schlenker}
\address{JMS: University of Luxembourg, Campus Kirchberg, Mathematics Research Unit, BLG, 6 rue Richard Coudenhove-Kalergi, L-1359 Luxembourg. {\tt jean-marc.schlenker@math.uni.lu}}


\begin{abstract}
We study the circulant complex Hadamard matrices of order $n$ whose entries are $l$-th roots of unity. For $n=l$ prime we prove that the only such matrix, up to equivalence, is the Fourier matrix, while for $n=p+q,l=pq$ with $p,q$ distinct primes there is no such matrix. We then provide a list of equivalence classes of such matrices, for small values of $n,l$.
\end{abstract}

\maketitle


\section{Introduction and results}

A complex Hadamard matrix of order $n$ is a matrix $H$ having as entries complex numbers of modulus $1$, such that $H/\sqrt{n}$ is unitary. Among complex Hadamard matrices, those with all entries roots of unity are said to be of Butson type. The basic example here is the Fourier matrix, $F_n=(w^{ij})_{ij}$ with $w=e^{2\pi i/n}$:
$$F_n=\begin{pmatrix}
1&1&1&\cdots&1\\
1&w&w^2&\cdots&w^{n-1}\\
\cdots&\cdots&\cdots&\cdots&\cdots\\
1&w^{n-1}&w^{2(n-1)}&\cdots&w^{(n-1)^2}
\end{pmatrix}.$$

We denote by $C_n(l)$ the set of complex Hadamard matrices of order $n$, with all entries being $l$-th roots of unity. As a first example here, observe that we have $F_n\in C_n(n)$.

In general, the complex Hadamard matrices are known to have applications to a wide array of questions, ranging from electrical engineering to von Neumann algebras and quantum physics, and the Butson ones are known to be at the ``core'' of the theory. For further details here, we recommend the excellent survey article by Tadej and \.Zyczkowski \cite{tzy}, and the subsequent website, made and maintained in collaboration with Bruzda\footnote{{\tt http://chaos.if.uj.edu/pl/$\widetilde{\ }$karol/hadamard}}. 

We are more specifically interested in understanding the complex Hadamard matrices of Butson type which are {\it circulant}, that is, of the form $(H_{ij})_{i,j=1,\ldots,n}$ with $H_{ij}$ depending only on $i-j$. We denote by $C^{circ}_n(l)$ the set  of circulant matrices in $C_n(l)$, and by $C^{circ,1}_n(l)$ the set of matrices from $C^{circ}_n(l)$ with 1 on the diagonal.

Regarding the motivations for the study of such matrices, let us mention: (1) their key importance for the construction of complex Hadamard matrices, see \cite{tzy}, (2) their relation with cyclic $n$-roots and their applications, see \cite{bjo}, and (3) their relation with the Circulant Hadamard Conjecture, a beautiful mathematical problem, to be explained now.

In the real case -- that is, for $l=2$ -- only one example is known, at $n=4$, namely the $4\times 4$ matrix with 
$-1$ on the diagonal and $1$ elsewhere.
For larger values of $n$ it is conjectured that there is no example:

\begin{conjecture}[Circulant Hadamard Conjecture] \label{chc}
$C^{circ}_n(2)=\emptyset$ for all $n>4$.
\end{conjecture}

For larger values of $l$, however, things are different and, to a large extent, mysterious. Our main goal here is to make some progress on the understanding of the values of $n,l$ which allow the existence of  circulant Hadamard matrices of Butson type. 

It is known since \cite{bac}, \cite{fau} that $F_n$ is equivalent to a circulant matrix, and it follows that $C^{circ}_n(n)\neq \emptyset$. We will prove below that when $n$ is prime, the Fourier matrix (in circulant form) is the only example, that is, any matrix in $C^{circ}_n(n)$ is equivalent to the Fourier matrix.

\begin{theorem} \label{tm:uniqueness-p}
Any matrix in $C^{circ}_p(p)$ is equivalent to the Fourier matrix $F_p$.
\end{theorem}

Here, as in the rest of the paper, the equivalence relation which is considered among circulant matrices is the cyclic permutation of rows and columns, and multiplication of all entries by a constant. However, since the Fourier matrix is not circulant, we need a wider notion of equivalence in the statement of this theorem. More precisely, we use the standard notion of equivalence for (non-circulant) complex Hadamard matrices, that is, permutation of the rows and columns, and multiplication of each row and each column by a constant.

Our second result is a non-existence statement for circulant Butson matrices based on two different prime numbers:

\begin{theorem} \label{tm:pq}
Let $p,q\geq 5$ be two distinct primes. Then $C^{circ}_{p+q}(pq)=\emptyset$.
\end{theorem}

A large part of the proof is actually valid for non-circulant matrices. 
To put this result in perspective, recall the following result \cite[Theorem 7.9]{bbs}. Assume $C_n(l)\neq\emptyset$.
\begin{enumerate}
\item If $n=p+2$ with $p\geq 3$ prime, then $l\neq 2p^b$.
\item If $n=2q$ with $p>q\geq 3$ primes, then $l\neq 2^ap^b$.
\end{enumerate}
It follows from the first result that for $p\geq 3$ prime,  $C_{p+2}(2p)=\emptyset$, 
so $C^{circ}_{p+2}(2p)=\emptyset$.  However it remains unclear whether $C_{p+3}(3p)=\emptyset$ for $p\geq 5$ prime. 
It also seems natural to ask whether, for $p,q$ be two distinct primes, $C_{p+q}(pq)=\emptyset$.

The last main contribution of the present paper is an ``experimental'' study 
of the set of equivalence classes of matrices in $C^{circ}_n(l)$, for small
values of $n,l$. See Section \ref{sc:classification}. This study was done using a computer program, designed to find all possible equivalence classes of circulant matrices which can contain a matrix of Butson type. At the counting level, the main results are presented in Table \ref{tb:table}.

The classification results obtained by this experimental approach lead to a number of quite natural questions, for instance the following ones:
\begin{itemize}
\item What are the obstructions to the existence of a matrix in  $C^{circ}_n(l)$ for various values of $n,l$? In the non-circulant case, there are a few obstructions known, which explain quite well, for small $n,l$, when $C_n(l)=\emptyset$. In the circulant  case, however, few obstructions are known beyond those known for general (non-circulant) Butson matrices. 

\item For $p$ prime and $k\in {\mathbb N}$, does $C^{circ}_p(kp)$ contain any matrix  other than those obtained from the Fourier matrix $F_p$?

\item What is the number of elements of $C^{circ}_4(l)$,  $C^{circ}_8(l)$,  $C^{circ}_9(l)$, with $l\in {\mathbb N}$ arbitrary? For instance, what are the next terms in the sequence defined as the number of equivalence classes in $C_9^{circ}(3k)$, with first terms $6, 24, 62, 108, 172$?
\end{itemize}
Finding answers to these questions might ultimately lead to progress on Conjecture \ref{chc}.

\bigskip

\noindent \textbf{Acknowledgements.} 

The authors are particularly grateful to Teo Banica and Ion Nechita, who participated in earlier parts of
this project and made valuable contributions to the results presented here.

Theorem \ref{tm:p} was obtained thanks to the website {\it MathOverflow}. 
Shortly after we posted a question\footnote{{\tt http://mathoverflow.net/questions/135949}} which is a more elementary but equivalent statement of Theorem \ref{tm:p}, Noam Elkies posted an answer to this question which is basically a proof of Theorem \ref{tm:p}. Peter M\"uller then posted another answer pointing to the papers by Gluck, by R\'onyai and Sz\H{o}nyi, and by Hiramine, containing the ``original'' proofs of this result. We are grateful to  Noam Elkies and Peter M\"uller for their contributions.

The authors are also grateful to two anonymous referees for helpful suggestions on an earlier version
of this work.

\section{Circulant Butson matrices}
\label{sc:background}

We consider in this paper various $n\times n$ matrices over the real or complex numbers. The matrix indices will range in $\{0,1,\ldots,n-1\}$, and will be taken modulo $n$.

\begin{definition}
A complex Hadamard matrix is a square matrix $H\in M_n(\mathbb C)$, all whose entries are on the unit circle, $|H_{ij}|=1$,  and whose rows are pairwise orthogonal.
\end{definition}

It follows from basic linear algebra that the columns are pairwise orthogonal as well. In fact, the $n\times n$ complex Hadamard matrices form a real algebraic manifold, given by:
$$C_n(\infty)=M_n(\mathbb T)\cap\sqrt{n}U(n)$$

Here, and in what follows, we denote by $\mathbb T$ the unit circle in the complex plane.

The basic example is the Fourier matrix, $F_n=(w^{ij})_{ij}$ with $w=e^{2\pi i/n}$. There are many other interesting examples, see \cite{tzy}. These examples are often constructed by using roots of unity, and we have the following definition, going back to Butson's work \cite{but}:

\begin{definition}
The Butson class $C_n(l)$ with $l\in\{2,3,\ldots,\infty\}$ consists of the $n\times n$ complex Hadamard matrices having as entries the $l$-th roots of unity. In particular:
\begin{enumerate}
\item $C_n(2)$ is the set of usual (real) $n\times n$ Hadamard matrices.

\item $C_n(4)$ is the set of $n\times n$ Turyn matrices, over $\{\pm 1,\pm i\}$.

\item $C_n(\infty)$ is the set of all $n\times n$ complex Hadamard matrices.
\end{enumerate}
\end{definition}

As explained in the introduction, we will be mostly interested in the circulant case, with the Circulant Hadamard Conjecture (CHC) in mind. Observe that the CHC states that, for $n>4$:
$$C_n^{circ}(2)=\emptyset.$$ 

\section{Circulant matrices of prime order}

We fix a prime number $p$, and we let $\omega=e^{2\pi i/p}$. The proof of Theorem \ref{tm:uniqueness-p} is based on:

\begin{theorem} \label{tm:p}
Let $M=(m_{i,j})_{i,j=1}^n \in M_p({\mathbb Z}/p{\mathbb Z})$ be a circulant matrix. Then $U=(\omega^{m_{ij}})_{i,j=0}^{n-1}$ is a circulant complex Hadamard matrix if and only if its first row is given by a polynomial of degree $2$, that is, if and only if there are $a,b,c\in Z/p{\mathbb Z}$ with $a\neq 0$ such that, for all $j\in {\mathbb Z}/p{\mathbb Z}$, 
$$ m_{0j} = aj^2 + bj + c~. $$ 
\end{theorem}

We first reformulate this theorem in a more combinatorial but clearly
equivalent form:

\begin{theorem} \label{tm:equiv}
Let $p\geq 3$ be a prime number, and let $u:\mathbb{Z}/p\mathbb{Z}\to \mathbb{Z}/p\mathbb{Z}$ be a map such that, for all $l\in \mathbb{Z}/p\mathbb{Z}$,$l\neq 0$, the map $k\mapsto u(k+l)-u(k)$ is a permutation. 
Then $u$ is a polynomial of degree $2$. 
\end{theorem}

Under this form, the statement is well-known to specialists of finite geometry, who call the above functions ``planar''. The fact that any planar function on a field of prime order is a quadratic polynomial was proved independently by Gluck \cite{glu}, by   R\'onyai and Sz\H{o}nyi \cite{rsz} and by Hiramine \cite{hir}. See the introduction/acknowledgements.

\begin{theorem} \label{tm:p-equiv}
Let $M=(m_{i,j})_{i,j=0}^{n-1} \in M_p({\mathbb Z}/p{\mathbb Z})$ be a circulant matrix. If the first row of $M$ is given by a polynomial of degree $2$
then $U=(\omega^{m_{ij}})_{i,j=0}^{n-1}$ is equivalent to the Fourier matrix.
\end{theorem}

\begin{proof}
Suppose first that the first row of $M$ is given by a polynomial of degree $2$. This means that there are elements $a,b,c\in {\mathbb Z}/p{\mathbb Z}$ with $a\neq 0$ such that 
$$ \forall j\in {\mathbb Z}/p{\mathbb Z}, \quad m_{0j} = aj^2 + bj +c~. $$
Since $M$ is circulant, this means that 
$$  \forall i,j\in {\mathbb Z}/p{\mathbb Z}, \quad m_{ij} = a(j-i)^2 + b(j-i) +c~. $$

First, we remove from each entry of $M$ the first entry of the same column. 
We get a new matrix $M'$ equivalent
to $M$, with entries given by 
\begin{eqnarray*}
m'_{ij} & = & m_{ij}-m_{0j} \\
& = &  a((j-i)^2-j^2) - bi \\
& = & ai^2-2aij-bi~.
\end{eqnarray*}

Second, we remove from each entry of this matrix $M'$ the first entry of its row. We obtain a new matrix $M''$ which is equivalent to $M'$ and therefore to $M$, with entries 
$$ m''_{ij} =  m'_{ij}-m'_{i0}=  -2aij~. $$

Third, we permute the rows by sending row $i$ to row $-2ai$. We finally obtain a matrix $M'''$, still equivalent to $M$, with entries
$$ m'''_{ij}=ij~. $$
Clearly this is the Fourier matrix, so $M$ is equivalent to $F_p$. 
\end{proof}

\section{Matrices based on two prime numbers}

We now turn to the proof of Theorem \ref{tm:pq}. Note that most of the
proof deals with non-circulant complex Hadamard matrices of Butson type.

In what follows, $p$ and $q$ denote two distinct primes. 
We consider a matrix in $C_{p+q}(pq)$, and use additive notations, so that
$M\in M_{p+q}({\mathbb Z}/pq{\mathbb Z})$. We call $L_i, i=0,\ldots, p+q-1$ the rows of $M$, 
and for $i,j\in {{\mathbb Z}/(p+q){\mathbb Z}}$ we set $L_{ij}=L_j-L_i$. 

The fact that $M$ corresponds to a Hadamard matrix translates as the following basic property.

\begin{lemma} \label{lm:cycles}
For all distinct $i,j\in {{\mathbb Z}/(p+q){\mathbb Z}}$, there exists a partition 
${{\mathbb Z}/(p+q){\mathbb Z}}=P_{ij}\sqcup Q_{ij}\sqcup R_{ij}$ and an element 
$r\in {\mathbb Z}/pq{\mathbb Z}$ such that:
\begin{itemize}
\item $\# R_{ij}=2$, and $L_{ij}(k)=r$ for all $k\in R_{ij}$,
\item $\# P_{ij}=p-1$, and $ \{ L_{ij}(k), k\in P_{ij}\} = \{ r+q({\mathbb Z}/pq{\mathbb Z})\}\setminus \{ r\}~, $
\item $\# Q_{ij}=q-1$, and $ \{ L_{ij}(k), k\in Q_{ij}\} = \{ r+p({\mathbb Z}/pq{\mathbb Z})\}\setminus \{ r\}~. $
\end{itemize}
\end{lemma}

\begin{proof}
This follows from the Chinese Remainder Theorem.
\end{proof}

Thus, the values of $L_{ij}$ on $P_{ij}$  go through all elements of a $p$-cycle in ${\mathbb Z}/pq{\mathbb Z}$ with the exception of $r$,  the values of $L_{ij}$ on the elements of $Q_{ij}$ 
go through all elements of a $q$-cycle in ${\mathbb Z}/pq{\mathbb Z}$ with the exception of $r$, and the value $r$ is taken twice, exactly at the two elements of $R_{ij}$.

To simplify a little the notations, we set $P^+_{ij}=P_{ij}\cup R_{ij}$ and 
$Q^+_{ij}=Q_{ij}\cup R_{ij}$. Note that the lemma above, as other statements below, also hold
if we swap $P$ and $Q$.

\begin{lemma} \label{lm:p}
Suppose that $i,j\in {\mathbb Z}/(p+q){\mathbb Z}$ are distinct and that 
$a,b\in {\mathbb Z}/(p+q){\mathbb Z}$ are distinct. Then $L_{ij}(b)-L_{ij}(a)\in q({\mathbb Z}/pq{\mathbb Z})$
if and only if $a,b\in P^+_{ij}$.
\end{lemma}

\begin{proof}
If $a,b\in Q_{ij}$, then $L_{ij}(b)-L_{ij}(a)\in p({\mathbb Z}/pq{\mathbb Z})$. But $p({\mathbb Z}/pq{\mathbb Z})\cap q({\mathbb Z}/pq{\mathbb Z})=\{ 0\}$ since $p,q$ are prime. So $L_{ij}(b)-L_{ij}(a)=0$, which is not possible if $a,b\in Q_{ij}$ are distinct. So either $a$ or $b$ is in $P^+_{ij}$. But then it follows from the definition of $P^+_{ij}$, and from the fact that $L_{ij}(b)-L_{ij}(a)\in q({\mathbb Z}/pq{\mathbb Z})$, that the other is in $P^+_{ij}$, too.
\end{proof}

\begin{lemma} \label{lm:PijQjk}
Let $i,j,k\in {\mathbb Z}/(p+q){\mathbb Z}$ be distinct. Then $\#(P_{ij}\cap Q^+_{jk})\leq 2$.
\end{lemma}

\begin{proof}
Suppose by contradiction that $P_{ij}\cap Q^+_{jk}$ contains three distinct
elements $a,b,c$. But then either two of them are in $P_{ij}\cap Q^+_{jk}\cap P^+_{ki}$ or two of them are in $P_{ij}\cap Q^+_{jk}\cap Q^+_{ki}$. We suppose for instance that $a,b\in P_{ij}\cap Q^+_{jk}\cap P^+_{ki}$, the same argument applies in the other cases.

Since $a,b\in P_{ij}$, $L_{ij}(b)-L_{ij}(a)\in q({\mathbb Z}/pq{\mathbb Z}$). Similarly, since $a,b\in P^+_{ki}$, $L_{ki}(b)-L_{ki}(a)\in q({\mathbb Z}/pq{\mathbb Z})$, so that $L_{jk}(b)-L_{jk}(a)=-(L_{ij}(b)-L_{ij}(a))-(L_{ki}(b)-L_{ki}(a))\in q({\mathbb Z}/pq{\mathbb Z})$. However we also know that $a,b\in Q^+_{jk}$, so that $L_{jk}(b)-L_{jk}(a)\in p({\mathbb Z}/pq{\mathbb Z})$, and therefore $L_{jk}(b)-L_{jk}(a)=0$, which contradicts the fact that $a,b\in P_{ij}$ and $a\neq b$.
\end{proof}

\begin{lemma} \label{lm:PijPjk}
Suppose that $p\geq 5$. Then for $i,j,k\in {\mathbb Z}/(p+q){\mathbb Z}$ distinct, $\#(P_{ij}\cap P_{jk})\geq 2$.
\end{lemma}

\begin{proof}
By definition,
$$ P_{ij} = P_{ij}\cap (P_{jk}\sqcup Q^+_{jk}) 
= (P_{ij}\cap P_{jk}) \sqcup (P_{ij}\cap Q^+_{jk})~. $$  
However $\# P_{ij}=p-1$, while $\#(P_{ij}\cap Q^+_{jk})\leq 2$ by Lemma \ref{lm:PijQjk}. 
The result follows.
\end{proof}

\begin{corollary} \label{cr:sub}
If $p\geq 5$ then for $i,j,k\in {\mathbb Z}/(p+q){\mathbb Z}$ distinct, $P^+_{ij}\cap P^+_{jk}\subset P^+_{ik}$.
\end{corollary}

\begin{proof}
According to Lemma \ref{lm:PijPjk}, $\#(P^+_{ij}\cap P^+_{jk})\geq 2$. Moreover if $a,b\in P^+_{ij}\cap P^+_{jk}$ are distinct, then $L_{ij}(b)-L_{ij}(a)\in q({\mathbb Z}/pq{\mathbb Z})$ and $L_{jk}(b)-L_{jk}(a)\in q({\mathbb Z}/pq{\mathbb Z})$. It follows that $L_{ik}(b)-L_{ik}(a)\in q({\mathbb Z}/pq{\mathbb Z})$ and therefore, from Lemma \ref{lm:p}, that $a$ and $b$ are in $P^+_{ik}$.
\end{proof}

\begin{lemma} \label{lm:PijQjk2}
Suppose that $p,q\geq 5$, and let $i,j,k\in {\mathbb Z}/(p+q){\mathbb Z}$ be distinct. 
Then $P_{ij}\cap Q_{jk}=\emptyset$. As a consequence, $P_{ij}\subset P_{jk}^+$ and $Q_{ij}\subset Q_{jk}^+$.
\end{lemma}

\begin{proof}
Let $a\in P_{ij}\cap P_{jk}$ and let $b\in Q_{ij}\cap Q_{jk}$ --- such elements
exist by Lemma \ref{lm:PijPjk}. We know that $a\in P^+_{ik}, b\in Q^+_{ik}$ 
by Corollary \ref{cr:sub}. 

Suppose now by contradiction that $c\in P_{ij}\cap Q_{jk}$. Then $a,c\in P_{ij}$
so that 
$$ L_{ij}(c) - L_{ij}(a)\in q({\mathbb Z}/pq{\mathbb Z})~, $$
and $a\in P_{jk},c\in Q_{jk}$, so that
$$ L_{jk}(c) - L_{jk}(a)\not\in q({\mathbb Z}/pq{\mathbb Z})~. $$
So 
$$ L_{ik}(c)-L_{ik}(a) = (L_{jk}(c) - L_{jk}(a)) + (L_{ij}(c) - L_{ij}(a)) 
\not\in q({\mathbb Z}/pq{\mathbb Z})~. $$
Since $a\in P^+_{ik}$, it follows that $c\not\in P^+_{ik}$. 

The same argument with $b$ instead of $a$ shows that $c\not\in Q^+_{ik}$,
which is a contradiction.

This shows that $P_{ij}\subset P_{jk}^+$. It follows that 
$P_{ij}\subset  P_{ij}^+\cap P_{jk}^+$, and therefore that  $P_{ij}\subset  P_{jk}^+$
by $P_{ik}^+$ by Corollary \ref{cr:sub}. The same argument works for $Q_{ij}$.
\end{proof}

\begin{corollary} \label{cr:ijkl}
Suppose that $p,q\geq 5$. Then for all $i,j,k,l\in {\mathbb Z}/(p+q){\mathbb Z}$ such that 
$i\neq j$ and $k\neq l$,
$\#(P_{ij}\cap P_{kl})\geq p-3$ and $\#(Q_{ij}\cap Q_{kl})\geq q-3$.
\end{corollary}

\begin{proof}
We have $P_{ij}\subset P_{jk}^+$ and $P_{kl}\subset P_{jk}^+$ by Lemma \ref{lm:PijQjk2}.
The first result follows because $\# P_{ij}=\#P_{kl}=p-1$ while $\# P_{jk}^+=p+1$. 
The same argument used with $Q_{ij}$ gives the second result.
\end{proof}

We now prove a result about sets having large intersections with all their circular rotations. 

\begin{proposition}\label{prop:intersection}
Let $A \subset \mathbb Z / n\mathbb Z$ be a set with $\# A = a$ such that, for all $y \in \mathbb Z / n\mathbb Z$, $\#(A \cap (y+A)) \geq b$, for some fixed integer $b$. Then, we have that $a\geq b\sqrt n$.
\end{proposition}
\begin{proof}
Let $\pi:\mathbb Z / n\mathbb Z\to \{ 0,1\}$ be the indicator function of $A$,  $\pi(x)=1$ if $x\in A$,
$\pi(x)=0$ otherwise. Let $\pi':\mathbb Z / n\mathbb Z\to \{ 0,1\}$ be the indicator function of $-A$, $\pi'(x)=\pi(-x)$. 

In terms of indicator functions, the hypothesis about the cardinality of intersections reads
$$ \forall y \in\mathbb Z / n\mathbb Z, \quad  \sum_{x\in \mathbb Z / n\mathbb Z} \pi(x)\pi(x+y) \geq b~, $$
or in other terms $(\pi *\pi')(-y)\geq b$. Taking the value at $0$ of the Fourier transform and using the fact that 
$\hat{\pi'}=\overline{\hat \pi}$, we find that
$$ |\hat \pi(0)|^2 \geq b\sqrt n~. $$
However, we also know that 
$$ \|\hat\pi\|^2 = \| \pi\|^2 = \sum_{x\in \mathbb Z / n\mathbb Z} \pi(x)^2 = a~. $$
The conclusion $b \sqrt n \leq a$ follows from the previous two relations.
\end{proof}

We now have all the elements needed to prove Theorem \ref{tm:pq}.

\begin{proof}[Proof of Theorem \ref{tm:pq}]
Consider a matrix $C_{p+q}^{circ}(pq)$, written in additive notation as $M$ with coefficients in $\mathbb Z/pq\mathbb Z$,
and put, with the notations of Lemma \ref{lm:cycles}, $A = P_{01} \subset {\mathbb Z}/(p+q){\mathbb Z}$. 
Note that, for all $i \in {\mathbb Z}/(p+q){\mathbb Z}$, the sets $P_{i,i+1}$ are exactly 
the circular rotations of $A$: $P_{i,i+1} = i+A$, and, by Corollary \ref{cr:ijkl}, 
we have $ \#(P_{01} \cap P_{i,i+1}) \geq p-3$.

From Proposition \ref{prop:intersection}, with $a=\#A=p-1$, $b=p-3$ and $n=p+q$, it follows that
$$ p-1\geq \sqrt{p+q}(p-3)~, $$
and, since $q\geq 5$, that
$$ (p-1)^2 \geq (p-3)^2 (p+5)~. $$
However a direct examination shows that 
$$ (x-3)^2(x+5) - (x-1)^2 $$
is strictly positive for $x\geq 5$, and a contradiction follows.
\end{proof}

\section{Counting and classification results for small matrices}
\label{sc:classification}

There are a number of known obstructions to the existence of a matrix in $C_n(l)$:

\begin{theorem}
We have the following results:
\begin{enumerate}
\item Lam-Leung: If $C_n(l)\neq\emptyset$ and $l=p_1^{a_1}\cdots p_k^{a_k}$ then $n\in p_1\mathbb N+\ldots+p_k\mathbb N$.

\item de Launey: If $C_n(l)\neq\emptyset$ then there is $d\in\mathbb Z[e^{2\pi i/l}]$ such that $|d|^2=n^n$.

\item Sylvester: If $C_n(2)\neq\emptyset$ then $n=2$ or $4|n$.

\item Sylvester': If $C_n(l)\neq\emptyset$ and $n=p+2$ with $p\geq 3$ prime, then $l\neq 2p^b$.

\item Sylvester'': If $C_n(l)\neq\emptyset$ and $n=2q$ with $p>q\geq 3$ primes, then $l\neq 2^ap^b$.

\item Haagerup: If $C_5(l)\neq\emptyset$ then $5|l$.
\end{enumerate}
\end{theorem}

\begin{proof}
See respectively \cite{lle}, \cite{lau}, \cite{syl}, \cite{bbs}, \cite{bbs}, \cite{ha1}.
\end{proof}

In addition, there is a less explicit generalized Turyn obstruction, see \cite[Prop. 2.4]{bns}.

\begin{theorem}
If $C_5(l)\neq\emptyset$ then there exists $(a_1, a_2, \cdots, a_l)\in {\mathbb N}^l$ such that 
$\sum_i a_i=n$ and that 
$$ \sum_{i,k} a_i a_{i+k} w^k=n~, $$
where $w=e^{2\pi i/l}$.
\end{theorem}

{\small\begin{center}
\begin{table}
\begin{tabular}[t]{|r|c|c|c|c|c|c|c|c|c|c|c|c|c|c|c|}
\hline $n\backslash l$\!
&2&3&4&5&6&7&8&9&10&11&12&13&14&15\\
\hline 2&$\times_t$&$\times$&0&$\times$&$\times_t$&$\times$&0&$\times$&$\times_t$&$\times$&0&$\times$&$\times_t$&$\times$\\
\hline 3&$\times$&$F_3$&$\times$&$\times$&$(F_3)$&$\times$&$\times$&$(F_3)$&$\times$&$\times$&$(F_3)$&$\times$&$\times$&$(F_3)$\\
\hline 4&1&$\times$&2&$\times$&4&$\times$&5&$\times$&7&$\times$&8&$\times$&10&$\times$\\
\hline 5&$\times$&$\times$&$\times$&$F_5$&$\times_{pq}$&$\times$&$\times$&$\times$&$(F_5)$&$\times$& $\times_h$&$\times$&$\times$&$(F_5)$\\
\hline 6&$\times_s$&$\times_t$&$\times_t$&$\times$&$\times_t$&$\times$&0&$\times_t$&$\times_{s}$&$\times$&3&$\times$&$\times_{s}$&$\times_t$\\
\hline 7&$\times$&$\times$&$\times$&$\times$&0&$F_7$&$\times$&$\times$&$\times_{pq}$&$\times$&0&$\times$&$(F_7)$&$\times$\\
\hline 8&$\times_t$&$\times$&4&$\times$&$\times_t$&$\times$&9&$\times$&$\times_t$&$\times$&14&$\times$&0& $\times_{pq}$\\
\hline 9&$\times$&6&$\times$&$\times$&24&$\times$&$\times$&62&0&$\times$&108&$\times$&$\times_{s}$&172\\
\hline 10\!&$\times_s$&$\times$&0&$\times_t$&$\times_t$&$\times$&0&$\times$&$\times_t$&$\times$&0&$\times$&$\times$&0\\
\hline 11\!& $\times$ & $\times$ & $\times$ & $\times$ &$\times_t$&$\times$& $\times$ & $\times$ &0&$F_{11}$&$\times_t$& $\times$&0&0 \\
\hline 12\!&$\times_t$&0&$\times_t$& $\times$ &37&$\times$&0&0&$\times_t$& $\times$ && $\times$&$\times_t$& \\
\hline 13\!& $\times$ & $\times$ & $\times$ & $\times$ &0&$\times$& $\times$ & $\times$ &$\times_t$& $\times$&&$F_{13}$&$\times_t$& \\
\hline 14\!& $\times_s$ & $\times$ &$\times_t$& $\times$ &$\times_t$&0&$\times_t$& $\times$ &$\times_t$& $\times$&& $\times$&& \\ 
\hline 15\!& $\times$ & $\times_t$&$\times$ &$\times_t$&$\times_t$&$\times$ &$\times$ & $\times$ &$\times_t$& $\times$&& $\times$&$\times_t$& \\
\hline
\end{tabular}
\vskip4mm
\caption{{\em Existence and number of circulant Butson matrices}. Here $\times$ is the Lam-Leung obstruction, $\times_l,\times_h,\times_s$ are the de Launey, Haagerup and Sylvester obstructions, 
$\times_{pq}$ denote the Sylvester' and Sylvester'' obstructions, and $\times_t$ denotes the
generalized Turyn obstruction.}
\label{tb:table}
\end{table}
\end{center}}

Table \ref{tb:table} above describes for each $n,l$ either an obstruction to the existence of matrices in $C_n^{circ}(l)$, or ``0'' if no obstruction is known but there is no such matrix, or the number of equivalence classes of those matrices. The blank cells indicate that we have no results. 
Note that for many cells several different obstructions exist, in which case we chose one of the
possible obstructions.

The symbol ``$F_p$'' is used when $n=l$ is prime, so that Theorem \ref{tm:uniqueness-p} indicates that the only circulant Hadamard matrix is the Fourier matrix or one of its multiples (up to equivalence). The symbol
``$(F_p)$'' was used to indicate that the same uniqueness result applies but when $n$ is prime and  $l$ is a multiple of $n$, so that Theorem \ref{tm:uniqueness-p} does not apply directly. In those cases the uniqueness which is claimed results from computations.

A list, for small values of $n,l$ for which no obstruction is known, of the first rows of matrices in 
$C_n^{circ}(l)$, is available in an online appendix to the paper\footnote{\href{http://math.uni.lu/schlenker/programs/circbut/circbut.html}{http://math.uni.lu/schlenker/programs/circbut/circbut.html}}.

The classification results were obtained by using a computer program, which is also available in the online 
appendix. We stress that the search for matrices in $C^{circ}_n(l)$ cannot be done, except for very small
values of $n,l$, by a brutal search algorithm. The program used here, and presented in the companion 
webpage of the article, works for larger values of $n,l$ because it takes into account in a relatively 
efficient way both the structure of those matrices and the action of the symmetry group leaving those 
matrices invariant. We do not provide a detailed analysis of the way the algorithm works here, since the
program should be relatively easy to understand for specialists of complex Hadamard matrices.


\begin{thebibliography}{99}

\bibitem[ALM02]{alm}K.T. Arasu, W. de Launey and S.L. Ma, On circulant complex Hadamard matrices, {\em Des. Codes Cryptogr.} {\bf 25} (2002), 123--142.

\bibitem[Bac89]{bac}J. Backelin, Square multiples $n$ give infinitely many cyclic $n$-roots, Preprint (1989).

\bibitem[BBS09]{bbs}T. Banica, J. Bichon and J.-M. Schlenker, Representations of quantum permutation algebras, {\em J. Funct. Anal.} {\bf 257} (2009), 2864--2910.

\bibitem[BNS12]{bns}T. Banica, I. Nechita and J.-M. Schlenker, Analytic aspects of the circulant Hadamard conjecture, {\tt arXiv:1212.3589}.

\bibitem[Bj\"o90]{bjo}G. Bj\"orck, Functions of modulus $1$ on ${\rm Z}_n$ whose Fourier transforms have constant modulus, and cyclic $n$-roots, {\em NATO Adv. Sci. Inst. Ser. C Math. Phys. Sci.} {\bf 315} (1990), 131--140.

\bibitem[Bru65]{bru}R.A. Brualdi, A note on multipliers of difference sets, {\em J. Res. Nat. Bur. Standards} {\bf 69} (1965), 87--89.

\bibitem[But62]{but}A.T. Butson, Generalized Hadamard matrices, {\em Proc. Amer. Math. Soc.} {\bf 13} (1962), 894--898.

\bibitem[CK93]{ckh}R. Craigen and H. Kharaghani, On the nonexistence of Hermitian circulant complex Hadamard matrices, {\em Australas. J. Combin.} {\bf 7} (1993), 225--227.

\bibitem[Di\c t04]{dit}P. Di\c t\u a, Some results on the parametrization of complex Hadamard matrices, {\em J. Phys. A} {\bf 37} (2004), 5355--5374.

\bibitem[Fau01]{fau}J.-C. Faug\`ere, Finding all the solutions of Cyclic 9 using Gr\"obner basis techniques, {\em Lecture Notes Ser. Comput.} {\bf 9} (2001), 1--12.

\bibitem[Glu90]{glu}D. Gluck, A note on permutation polynomials and finite geometries, {\em Discrete Math.} {\bf 80} (1990), 97--100.

\bibitem[Haa97]{ha1}U. Haagerup, Orthogonal maximal abelian $*$-subalgebras of the $n\times n$ matrices and cyclic $n$-roots, in ``Operator algebras and quantum field theory'', International Press (1997), 296--323.

\bibitem[Haa08]{ha2}U. Haagerup, Cyclic $p$-roots of prime lengths $p$ and related complex Hadamard matrices, {\tt arXiv: 0803.2629}.

\bibitem[Hir89]{hir}Y. Hiramine, A conjecture on affine planes of prime order, {\em J. Comin. Theory Ser. A} {bf 52} (1989), 44--50.

\bibitem[LL00]{lle}T.Y. Lam and K.H. Leung, On vanishing sums of roots of unity, {\em J. Algebra} {\bf 224} (2000), 91--109.

\bibitem[Lau94]{lau}W. de Launey, On the non-existence of generalized weighing matrices, {\em Ars Combin.} {\bf 17} (1984), 117--132.

\bibitem[LS12]{lsc}K.H. Leung and B. Schmidt, New restrictions on possible orders of circulant Hadamard matrices, {\em Des. Codes Cryptogr.} {\bf 64} (2012), 143--151.

\bibitem[Ma84]{sma}S.L. Ma, On rational circulants satisfying $A^m=dI+\lambda J$, {\em Linear Algebra Appl.} {\bf 62} (1984), 155--161.

\bibitem[MW81]{mw1}J.H. McKay and S.S. Wang, The $v\times v$ (0,1,-1)-circulant equation $AA^T=vI-J$, {\em SIAM J. Algebraic Discrete Methods} {\bf 2} (1981), 266--274.

\bibitem[MW87]{mw2}J.H. McKay and S.S. Wang, On a theorem of Brualdi and Newman, {\em Linear Algebra Appl.} {\bf 92} (1987), 39--43.

\bibitem[RS89]{rsz}L. R\'onyai and T. Sz\H{o}nyi, Planar functions over finite fields, {\em Combinatorica} {\bf 9} (1989), 315--320.

\bibitem[Sch99]{sch}B. Schmidt, Cyclotomic integers and finite geometry, {\em J. Amer. Math. Soc.} {\bf 12} (1999), 929--952.

\bibitem[Syl867]{syl}J.J. Sylvester, Thoughts on inverse orthogonal matrices, simultaneous sign-successions, and tesselated pavements in two or more colours, with applications to Newton's rule, ornamental tile-work, and the theory of numbers, {\em Phil. Mag.} {\bf 34} (1867), 461--475.

\bibitem[TZ06]{tzy}W. Tadej and K. \.Zyczkowski, A concise guide to complex Hadamard matrices, {\em Open Syst. Inf. Dyn.} {\bf 13} (2006), 133--177.

\bibitem[Tur65]{tur}R.J. Turyn, Character sums and difference sets, {\em Pacific J. Math.} {\bf 15} (1965), 319--346.

\end{thebibliography}
\end{document}